\documentclass[10pt, 5p]{article} 

\usepackage[utf8]{inputenc} 
\usepackage{amsmath}
\usepackage{amsfonts}
\usepackage{amssymb}
\usepackage{amsthm}
\usepackage{a4wide}
\usepackage{booktabs}
\usepackage{url}
\usepackage[square,numbers]{natbib}
\usepackage{graphicx} 
\usepackage[parfill]{parskip} 
\newtheorem{lemma}{Lemma}

\newtheorem{theorem}{Theorem}
\everymath{\displaystyle}

\usepackage{amssymb} 
\everymath{\displaystyle}
\usepackage{xcolor}

\newtheorem{observation}{Observation}
\usepackage{amsmath}
\usepackage{amsfonts}
\usepackage{amssymb}
\usepackage{natbib}

\title{A MIP framework for non-convex uniform price day-ahead electricity auctions}

\date{January 2015}
\author{Mehdi Madani\footnote{Louvain School of Management - Place des Doyens 1 bte L2.01.01, 1348 Louvain-la-Neuve, Belgium. Email: mehdi.madani@uclouvain.be}         \and
        Mathieu Van Vyve\footnote{CORE \& Louvain School of Managment, 1348 Louvain-la-Neuve, Belgium. This author is also member of ECORE, the association between CORE and ECARES.}
}

\begin{document}

%

\vspace{-0.2cm}
{\let\newpage\relax\maketitle}

\begin{abstract}
It is well-known that a market equilibrium with uniform prices often does not exist in non-convex day-ahead electricity auctions. We consider the case of the non-convex, uniform-price Pan-European day-ahead electricity market "PCR" (Price Coupling of Regions), with non-convexities arising from so-called complex and block orders. Extending previous results, we propose a new primal-dual framework for these auctions, which has applications in both economic analysis and algorithm design. The contribution here is threefold. First, from the algorithmic point of view, we give a non-trivial exact (i.e. not approximate) linearization of a non-convex 'minimum income condition' that must hold for complex orders arising from the Spanish market, avoiding the introduction of any auxiliary variables, and allowing us to solve market clearing instances involving most of the bidding products proposed in PCR using off-the-shelf MIP solvers. Second, from the economic analysis point of view, we give the first MILP formulations of optimization problems such as the maximization of the traded volume, or the minimization of opportunity costs of paradoxically rejected block bids. We first show on a toy example that these two objectives are distinct from maximizing welfare. We also recover directly a previously noted property of an alternative market model. Third, we provide numerical experiments on realistic large-scale instances. They illustrate the efficiency of the approach, as well as the economics trade-offs that may occur in practice.



\end{abstract}


\maketitle

\newpage

\section{Introduction}
\label{intro}

\subsection{Equilibrium in non-convex day-ahead electricity auctions}\label{subsec:non-convex-markets}

An extensive literature now exists on non-convex day-ahead electricity markets or electricity pools, dealing in particular with market equilibrium issues in the presence of indivisibilities, see e.g. \cite{hogan,madani2,madani2014,oneill,oneill2007,ruiz2011,vanvyve,zak} and references therein. Research on the topic has been fostered by the ongoing liberalization and integration of electricity markets around the world during the past two decades. As electricity cannot be efficiently stored, non-convexities of production sets cannot be neglected, and bids introducing non-convexities in the mathematical formulation of the market clearing problem have been proposed for many years by most of power exchanges or electricity pools, allowing participants to reflect more accurately their operational constraints and cost structure.

It is now well known that due to these non-convexities, a market equilibrium with uniform prices may fail to exist (a single price per market area and time slot, no transfer payments, no losses incurred, and no excess demand nor excess supply for the given uniform market prices). To deal with this issue, almost all ideas proposed revolve around reaching back, or getting close to a convex situation where strong duality holds and shadow prices exist. For example, a now classic proposition in \cite{oneill} is to fix integer variables to optimal values for a welfare maximizing primal program whose constraints describe physically feasible dispatches of electricity, and compute multi-part equilibrium prices using dual variables of these fixing constraints. The same authors, in an unpublished working paper, have later adapted this proposition to the context of European power auctions, proposing to allow and compensate so-called paradoxically accepted block bids, thus deviating from a pure uniform price system. We review their proposition in section \ref{section-usingframework}. A recent proposition for electricity pools in \cite{ruiz2011} is to use a 'primal-dual' formulation (i.e. involving both executed quantities as primal variables, and market clearing prices as dual variables),  where 'getting close' is materialized by minimising the duality gap introduced by integer constraints, and where additional constraints  are added to ensure that producers are recovering their costs. The goal is to use uniform prices, while minimizing the inevitable deviation from market equilibrium, and providing adequate incentives to producers.

These last conditions, which have been used in Spain for many years, are usually called 'minimum income conditions' (MIC).
The natural way to model them is through imposing a lower bound on the revenue expressed as the product between executed quantities and market prices, yielding non-convex quadratic constraints. However, \cite{ruiz2009} propose an exact linearization of the revenue related to a set of bids of a strategic bidder participating to a convex market, relying on KKT conditions explicitly added to model the lower-level market clearing problem of a bilevel program, and linearized by introducting auxiliary binary variables.
 
Regarding the proposition in \cite{ruiz2011}, a market equilibrium exists only if the optimal duality gap is null, which is rarely the case with real instances, and the proposition choose  \emph{not} to enforce network equilibrium conditions (corresponding to optimality conditions of transmission system operators), nor that demand bids are not loosing money. The same remarks apply to \cite{GarciaBertrand}, where as non-convexities, only the minimum income conditions for producers are considered (not the indivisibilities).

Several other auction designs have been previously considered, which often propose to implement a non-uniform pricing scheme, see e.g. \cite{ruiz2011} for a review. 

The fact that equilibrium is not enforced for the convex part of the market clearing problem (i.e. for convex bids and the network model), is a key auction design difference compared with the choices made by power exchanges in Europe. In this article, we deal mainly with this last market model, which is further described in the next section, and illustrated in the toy example given in section \ref{subsec:toyexample}.

\subsection{The PCR market} \label{subsec:pcr}

We consider the Pan-European day-head electricity market being developed under the Price Coupling of Regions project (PCR), as publicly described in \cite{euphemia}. Essentially, it is a near-equilibrium auction mechanism using uniform prices, and where the only deviations from a perfect market equilibrium is the allowance of so-called "paradoxically rejected \emph{non-convex} bids" to which opportunity costs are hence incurred, because these bids, which are 'in-the-money', would be profitable for the computed clearing prices, see \cite{euphemia}. On the other side, all convex bids as well as TSOs must be 'at equilibrium' for the computed market clearing prices. This integrated market is coupling the CWE region (France, Belgium, Germany, the Netherlands, Luxembourg) with NordPool (Norway, Sweden, Denmark, Finland, Baltic countries), as well as Italy and OMIE (Spain, Portugal).

From the algorithmic point of view, when considering specifically the CWE region, we have previously shown that the market clearing problem can be restated as a MILP, without introducing any auxiliary variables to linearise the needed complementarity conditions modelling the near-equilibrium \cite{madani2014}. We also proposed a Benders-like decomposition procedure with locally strenghtened Benders cuts. Leaving aside a peculiar kind of bids from the Italian market (so-called PUN bids), introducing complex bids with a MIC condition yields a non-convex MINLP. The production-quality algorithm in use, EUPHEMIA (\cite{euphemia}), an extension of COSMOS previously used to clear the CWE market, is a sophisticated branch-and-cut algorithm handling all market requirements. However, due to the introduction of MIC bids, the algorithm is a heuristic, though COSMOS on which it relies is an exact branch-and-cut.

\subsection{Contribution and structure of this article} \label{subsec:contrib}

We provide here a new primal-dual framework for PCR-like auctions, which is mainly a continuation of ideas presented in \cite{madani2,madani2014}.  The objective is to present a unified approach to algorithmic and economic modelling issues concerning these European auctions, with useful computational applications. The approach essentially consists in using strong duality adequately to enforce complementarity conditions modelling equilibrium for the convex part of the market clearing problem, as required by European power exchanges. This approach is similar to the bilevel approach suggested in \cite{zak}, though we do not need to introduce any auxiliary variables to linearize quadratic terms when considering the equality of objective functions modelling optimality of the second level problem (i.e. equilibrium for the convex part). Moreover, we could choose to consider some additional continuous variables with clear economic interpretations as bounds on opportunity costs or losses of (non-convex) block bids.  This approach is used in \cite{madani2014} to derive a powerful Benders decomposition,   and also shows how to consider the case of piecewise linear bid curves yielding a QP setting, using strong duality for convex quadratic programs. The MIP framework proposed here is presented in section \ref{sec-framework}. It includes the following extensions.

First, so-called complex bids used in Spain and Portugal are added to the model, and we give an exact (i.e. not approximate) linearisation of a non-linear non-convex 'minimum income condition' (MIC) that must hold for these bids (\cite{euphemia}), which model revenue adequacy for producers. As mentioned above, this kind of conditions is provided for many years by the Spanish power exchange OMIE (\cite{omel}), and is also considered (in a different auction design setting) in \cite{GarciaBertrand,ruiz2011}. This enables us to give a MILP formulation of the PCR market clearing problem which avoids complementarity constraints and the use of any auxiliary variables, while taking into account these MIC conditions. This is developed in section \ref{section-mic}. 

Second, we show in section \ref{sec-framework} how to consider in the main MILP model, aside main decision variables such as prices and bid execution levels, additional variables which correspond to upper bounds on opportunity costs of block bids, and upper bounds on losses. Let us emphasize that current European market rules forbid paradoxically accepted block bids (executed bids incurring losses), and stating these particular conditions amount to requiring that some of the added variables must be null. This is used in section  \ref{section-usingframework} to develop economic analysis applications. For example, the framework can be used in particular  to provide the first (and reasonably tractable) MILP formulations of optimization problems such as the minimization of incurred opportunity costs, or the maximization of the traded volume. Let us note that in a convex context, no opportunity costs are incurred and any market clearing solution is welfare maximizing, so maximizing  the traded volume only amounts to  choose peculiar tie-breaking rules in case of indeterminacy. Yet it is shown on the toy example in the introductory section \ref{subsec:toyexample} that in a non-convex context,  these two objectives are both distinct from maximizing welfare.  To our knowledge, if opportunity costs of rejected block bids have been considered empirically in the past (e.g in \cite{Meeus2006,Meeus2009}), this point is new and could provide useful information to day-ahead auctions stakeholders. (We have presented partial results  about opportunity costs in a simplified setting at the EEM 14 conference, see \cite{madani2}.) Also, still using the added variables, we recover and present in a novel way properties of an alternative market model proposed in \cite{oneill2007}.

Finally, numerical experiments done on realistic large-scale instances are presented in section \ref{sec-numerical}. They show that our proposition allows to solve up to optimality market clearing instances with MIC bids, which correspond to the Spanish market design. 
This is the first time that real-life instances of this type of problems are solved to optimality. 
This straightforward approach does not behave as well for instances including both MIC and block bids.
However, a simple heuristic approach already yields provably high-quality solutions.
Regarding the economic analysis applications, results presented illustrate the trade-offs that may occur for realistic large-scale instances, for example between optimizing welfare and optimizing the traded volume. Again this is the first time that optimal solutions for these problems can be computed apart form toy examples of small sizes.

\section{A new primal-dual framework}\label{sec-framework}

Below, section \ref{subsec:toyexample} describes the context and issues for day-ahead markets with indivisibilities such as in the CWE region (with block bids). Then, section \ref{subsec:unrestricted-welfare} introduces the welfare maximization problem without equilibrium restrictions, i.e. neglecting adequate market clearing prices existence issues. Section \ref{subsec:duality-uniformprices} derives a related  dual program parametrized by the integer decisions, and several important economic interpretations relating dual variables, uniform prices and deviations from a perfect market equilibrium (losses and opportunity costs of executed/rejected non-convex bids). Finally, section \ref{subsec:newprimaldualframework} presents the basis of the new primal-dual framework proposed.

\subsection{Uniform prices and price-based decisions in the CWE region: a toy example}\label{subsec:toyexample}

We use here a toy example (\cite{madani2}) illustrating two key points. First, a market equilibrium may not exist in the presence of indivisible orders. Second, under European market rules where paradoxically rejected non-convex bids are allowed, a welfare maximizing solution is not necessarily a traded volume maximizing solution nor it is necessarily an opportunity costs minimizing solution. The toy example consists in a market clearing instance involving two demand continuous bids (e.g. two steps of a stepwise demand bid curve), and two sell block bids. Parameters are summarized in Table \ref{table1} :

\begin{table}[h]
\begin{center}
\begin{tabular}{|c|c|c|}
\hline 
Bids & Power (MW) & Limit price  (EUR/MW) \\ 
\hline 
A:  Buy bid 1 & 11 & 50 \\ 
\hline 
B:  Buy bid 2 & 14 & 10 \\ 
\hline 
C:  Sell block bid 1 & 10 & 5 \\ 
\hline 
D:  Sell block bid 2 & 20 & 10 \\ 
\hline 
\end{tabular} 
\end{center}
\caption{Toy market clearing instance}
\label{table1}
\end{table}

First, obviously, it is not possible to execute both sell block bids, as they offer a total amount of power of 30 MW, while the total demand is at most 25 MW. As they are indivisible, if there is a trade, either (i) bid C is fully executed or (ii) bid D is fully executed. Second, at equilibrium, by definition, for the given market prices, no bidder should prefer another level of execution of its bid. In particular, in-the-money (ITM) bids must be fully executed, out-of-the-money (OTM) bids must be  fully rejected, and fractionally executed bids must be right at-the-money (ATM). 

So in the first case (i), A is partially accepted and set the market clearing price to 50 EUR/MW, if any equilibrium with uniform prices exists. But in that case, block bid D is rejected while ITM: an opportunity cost of $20(50-10) = 800$ is incurred. This situation is accepted under the near-equilibrium European market rules described above. A direct computation shows that the welfare is then equal to 10(50-50) + 10(50-5) = 450, while the traded volume is 10. Similar computations in the case (ii) yield the  market outcome summarized in Table \ref{table2}. \


\begin{table}[h]
\begin{center}
\begin{tabular}{|c|c|c|c|c|}
\hline 
• & Price & Traded Volume & Welfare & Opportunity costs \\ 
\hline 
Matching C & 50 & 10 & 450 & 800 \\ 
\hline 
Matching D & 10 & 20 & 440 & 50 \\ 
\hline 
\end{tabular} 
\end{center}
\caption{Market outcome}
\label{table2}
\end{table}

In this toy example, case (i) maximizes welfare but generates (much) more opportunity costs and half the traded volume.

\subsection{Unrestricted welfare optimization}\label{subsec:unrestricted-welfare}
\label{sec:notation}

We formulate here the classical welfare optimization problem with an abstract and very general power transmission network representation that is still linear. 
It covers e.g. DC network flow models or the so-called ATC and Flow-based models used in PCR (see \cite{euphemia}). The usual network equilibrium conditions involving locational market prices apply, see \cite{madani2014}.

Binary variables $u_c$ are introduced to model the conditional acceptance of a set of hourly bids $hc \in H_c$, controlled via constraints (\ref{primal-eq3}). The conditional acceptance relative to a minimum income condition is dealt with in section \ref{section-mic}.

\section*{Notation}

Notation used throughout the text is provided here for quick reference. The interpretation of any other symbol is given within the text itself.

\medskip
\noindent
Sets and indices:

\begin{tabular}{ll}
$i$ & Index for hourly bids, in set $I$ \\ 
$j$ & Index for block bids, in set $J$ \\ 
$c$ & Index for MIC bids, in set $C$ \\
$hc$ & Index for hourly bids associated to the MIC bid $c$, in set $H_c$ \\
$l$ & Index for locations, $l(i)$ (resp. $l(hc)$) denotes the location of bid $i$ (resp. $hc$)\\
$t$ & Index for time slots, $t(i)$ (resp. $t(hc)$) denotes the time slot of bid $i$, (resp. $hc$) \\
$I_{lt} \subseteq I$ & Subset of hourly bids associated to location $l$ and time slot $t$ \\
$HC_{lt} \subseteq HC$  & Subset of MIC hourly suborders, associated to location $l$ and time slot $t$ \\
$J_l \subseteq J$ & Subset of block bids associated to location $l$ 
\end{tabular} 

\bigskip

Parameters:

\begin{tabular}{ll}
$P_i, P_{hc}$ & Power amount of hourly bid $i$ (resp. $hc$), \\ &  $P<0$ for sell bids, and $P>0 $ for demand bids \\
$P^t_j$ & Power amount of block bid $j$ at time $t$, same sign convention \\
$\lambda^i, \lambda^{hc}$ & Limit bid price of hourly bid $i$, $hc$\\
$\lambda^j$ & Limit bid price of block bid $j$ \\
$a_{m,k}$ & Abstract linear network representation parameters 
\end{tabular} 

\bigskip

Primal decision variables:

\begin{tabular}{ll}
$x_i \in [0,1]$ & fraction of power $ P_i $ which is executed \\
$x_{hc} \in [0,1]$ & fraction of power $P_{hc}$ (related to the MIC bid $c$) which is executed\\
$y_j \in \{0,1\}$ & binary variable which determines if the quantities  $P^t_j$ are fully accepted or rejected \\
$u_c \in \{0,1\}$  & binary variable controling the execution or rejection of the MIC bid $c$ \\ & (i.e. of the values of $x_{hc}$) \\
$n_k$ & variables used for the abstract linear network representation, related to net export positions \\
\end{tabular} 

\bigskip

Dual decision variables:

\begin{tabular}{ll}
$\pi_{lt}$ & uniform price (locational marginal price) for power in location $l$ and time slot $t$ \\
$v_m \geq 0$ & dual variable pricing the network constraint $m$, \\
$s_{i} \geq 0$ & dual variable interpretable as the surplus associated to the execution of bid $i \in I$\\
$s_{j} \geq 0$ & dual variable interpretable as the surplus associated to the execution of bid $j \in J$\\
$s_{hc} \geq 0$ & dual variable interpretable as the (potential) surplus associated to the execution of bid $hc$\\
$s_{c} \geq 0$ & dual variable interpretable as the surplus associated to the execution of the MIC bid $c$\\

\end{tabular}

\begin{multline}\label{primal-obj}
\max_{x,y,u,n} \sum_{i} (\lambda^{i}P^{}_{i})x_i + \sum_{c, h\in H_c} (\lambda^{hc}P^{}_{hc})x_{hc}  +  \sum_{j,t} (\lambda^{j}P^{t}_{j})y_j 
\end{multline}

subject to:

\begin{align}
&x_i \leq 1 & \forall i \in I \ & [s_i]\label{primal-eq1} \\
&y_j \leq 1 & \forall j \in J \ & [s_j]\label{primal-eq2} \\
&x_{hc} \leq u_{c} & \forall h \in H_c, c\in C \ & [s_{hc}] \label{primal-eq3} \\
&u_{c} \leq 1 &\forall c \in C &[s_{c}] \label{primal-eq4} \\
&\sum_{i \in I_{lt}}P^{}_{i}x_i + \sum_{j \in J_l}P^{t}_{j}y_j  + \sum_{hc\in HC_{lt}} P^{}_{hc}x_{hc} \nonumber\\ & \hspace{5cm}   = \sum_{k} e^k_{l,t} n_k, \ &  \forall (l,t) \ \   & [\pi_{l,t}]\label{primal-eq5} \\
&\sum_{k} a_{m,k} n_k  \leq w_{m}\ & \forall m \in N \ \ & [v_{m}]\label{primal-eq6}  \\                  
&x, y, u\geq 0, \label{primal-eq7} \\
&y,u \in \mathbb{Z} \label{primal-eq8}
\end{align}
Constraint \eqref{primal-eq5} is the balance equation at location $l$ at time $t$. Constraint \eqref{primal-eq6} is the capacity constraint of network element $m$. 

\subsection{Duality, uniform prices and opportunity costs}\label{subsec:duality-uniformprices}\label{subsec-duality-econ}

Let us now consider partitions $J=J_r \cup J_a$, $C=C_r \cup C_a$, and the following constraints, fixing all integer variables to some arbitrarily given values:

\begin{align}
&-y_{j_a} \leq -1 \qquad & \forall j_a  \in J_a \subseteq J &  [d^a_{j_a}] \label{fix-eq1}  \\
& y_{j_r} \leq 0 \qquad & \forall j_r \in J_r \subseteq J &  [d^r_{j_r}] \label{fix-eq2} \\
&- u_{c_a} \leq -1 \qquad & \forall c_a \in C_a\subseteq C & [du^a_{c_a}] \label{fix-eq3} \\
&u_{c_r} \leq 0 \qquad & \forall c_r \in C_r \subseteq C &[du^r_{c_r}] \label{fix-eq4}
\end{align}

Dropping integer constraints (\ref{primal-eq8}) not needed any more, this yields an LP whose dual is:

\begin{multline}\label{dual-obj}
\min \sum_{i} s_i + \sum_{j} s_j +\sum_c s_c + \sum_m w_m v_m  - \sum_{j_a \in J_a} d^a_{j_a} - \sum_{c_a \in C_a} d^a_{c_a}  
\end{multline}

\text{subject to: }
\begin{align}
&s_i + P^{}_i \pi_{l(i),t(i)} \geq P^{}_i \lambda^{i} , &  \forall i  \qquad [x_i] \label{dual-eq1} \\
&s_{hc} + P^{}_{hc}\pi_{l(hc),t(hc)} \geq P^{}_{hc}\lambda^{hc}, & \forall h \in H_c, c \ [x_{hc}] \label{dual-eq2} \\
&s_{j_r} + d^r_{j_r}+ \sum_t P^{t}_{j_r}\pi_{l(j_r),t} \geq P^{}_{j_r}\lambda^{j_r}, & \forall j_r \in J_r [y_{j_r}] \label{dual-eq3} \\
&s_{j_a} -  d^a_{j_a} + \sum_t P^{t}_{j_a}\pi_{l(j_a),t} \geq P^{}_{j_a}\lambda^{j_a}, & \forall j_a \in J_a \ [y_{j_a}] \label{dual-eq4} \\
&s_{c_r} + du^r_{c_r} \geq \sum_{h \in H_{c_r}} s_{hc_r}, &  \forall c_r \in C_r  \ [u_{c_r}] \label{dual-eq5} \\
&s_{c_a} - du^a_{c_a} \geq \sum_{h \in H_{c_a}} s_{hc_a}, &  \forall c_a \in C_a  \ [u_{c_a}] \label{dual-eq6}  \\
&\sum_m a_{m,k} v_m- \sum_{l,t}e^k_{l,t} \pi_{l,t} = 0 & \forall k \in K \ [n_k] \label{dual-eq7}\\
&s_i, s_j,s_c, s_{hc}, d^r_{j_r}, d^a_{j_a},du^r_{c_r},du^a_{c_a}, v_m \geq 0 \label{dual-eq8}
\end{align}

We now write down the complementarity constraints corresponding to these primal an dual programs parametrized by the integer decisions. Economic interpretations are stated afterwards:

\begin{align}
&s_i(1-x_i)=0 &\ \forall i \in I \label{cc-eq1}\\
&s_{j}(1-y_{j})=0 &\ \forall j \in J \label{cc-eq2}\\
&s_{hc}(u_{c}-x_{hc})=0 &\ \forall h,c \label{cc-eq3}\\
&s_{c}(1-u_{c})=0 &\forall c \in C \label{cc-eq4} \\
&v_m(\sum_k a_{m,k} n_k - w_m)=0 &\ \forall m \in N \label{cc-eq5}\\
&(1-y_{j_a})d^a_{j_a}=0 &\ \forall j_a \in J_a \label{cc-eq9}\\
&y_{j_r}d^r_{j_r}=0 &\ \forall j_r \in J_r \label{cc-eq8} \\
&(1-u_{c_1})du^a_{c_1}=0 & \forall c_1 \in C_1 \label{cc-eq7} \\
&u_{c_r}du^r_{c_r}=0 & \forall c_r \in C_r  \label{cc-eq6} \\
&x_i(s_i  + P^{}_{i}\pi_{l(i),t(i)} - P^{}_{i}\lambda^{i})=0 &\ \forall i \in I \label{cc-eq10}\\
&x_{hc}(s_{hc}  + P^{}_{hc}\pi_{l(hc),t(hc)} - P^{}_{hc}\lambda^{hc})=0 &\ \forall h,c \label{cc-eq11}\\
&y_{j_r}(s_{j_r} + d^r_{j_r} + \sum_t P^{t}_{j_r}(\pi_{l(j_r),t} - \lambda^{j_r}))=0 & \forall j_r \in J_r \label{cc-eq12} \\
&y_{j_a}(s_{j_a} - d^a_{j_a} + \sum_t P^{t}_{j_a}(\pi_{l(j_a),t} - \lambda^{j_a}))=0 & \forall j_a \in J_a \label{cc-eq13} \\
&u_{c_r}(s_{c_r} + du^r_{c_r} - \sum_{h \in H_{c_r}} s_{hc_r} ) =0 & \forall c_r \in C_r \label{cc-eq14} \\
&u_{c_a}(s_{c_a} - du^a_{c_a} - \sum_{h \in H_{c_a}} s_{hc_a} ) =0 & \forall c_a \in C_a \label{cc-eq15}
\end{align}

\begin{lemma}[Economic interpretation of $d^a, d^r$ \cite{madani2}] Take a pair of points $(x,y,u,n)$ and $(s,\pi_{l,t},d^a,d^r, du^a, du^r)$ respectively satisfying primal conditions (\ref{primal-eq1})-(\ref{fix-eq4}) and dual conditions (\ref{dual-eq1})-(\ref{dual-eq8}), such that complementarity constraints (\ref{cc-eq1})-(\ref{cc-eq15}) are satisfied. For the uniform prices $\pi_{l,t}$: (i) $d^a_{j_a}$  is an upper bound on the actual loss (if any) $\min [0, \sum_t P^{t}_{j_a}(\lambda^{j_a}-\pi_{l(j_a),t})]$ of the executed block order $j_a$, (ii) $d^r_{j_r}$ is an upper bound on the opportunity cost $\max [0, \sum_t P^{t}_{j_r}(\lambda^{j_r}-\pi_{l(j_r),t})]$ of the rejected order $j_r$. \label{lemma-dadr} \end{lemma}

\begin{proof}
(i) Conditions (\ref{cc-eq13}) show that for an accepted block $y_{j_a}=1$, we have $s_{j_a} - d^a_{j_a} = \sum_t P^{t}_{j_a}(\lambda^{j_a}-\pi_{l(j_a),t})$, the right-hand side corresponding to the gain (if positive) or loss (if negative) of the bid. As $s_{j_a} \geq 0$, the loss is bounded by $d^a_{j_a}$.

(ii) For a rejected block bid, $y_{j_r}=0$, and conditions (\ref{cc-eq2}) imply $s_{j_r} = 0$, which used in dual conditions (\ref{dual-eq3}) directly yields the result, as $d^r_{j_r} \geq 0$.
\end{proof}


The following lemma proposes analogous interpretations for the case of MIC orders. Intuitively, neglecting for now the so-called MIC condition, the shadow cost of forcing a MIC order to be rejected (given by $du^r$)  is at least equal to the sum of all maximum missed surpluses generated by its hourly suborders at the given market prices, while there is no 'shadow cost' at forcing it to be accepted (its suborders are then cleared using standard rules for hourly bids):

\begin{lemma}[Interpretation of $du^a, du^r$] \label{lemma-duadur}
(i) $du^r_{c_r}$ is an upper bound on the sum of maximum missed individual hourly surpluses $\sum_{h\in H_{c_r}} ( \max[0,P^{}_{hc_r}(\lambda^{hc_r} - \pi_{l(hc_r),t(hc_r)})])$ of the rejected MIC order $c_r$. (ii) We can always assume $du^a_{c_a}=0, \forall c_a \in C_a $.
\end{lemma}

\begin{proof}
(i) Conditions of type (\ref{cc-eq4}) show that $s_{c_r}=0$, while (\ref{dual-eq2}) and (\ref{dual-eq8}) give $s_{hc_r} \geq \max [0, P^{}_{hc_r}(\lambda^{hc_r} - \pi_{l(hc_r),t(hc_r)})]$. Using these two facts in (\ref{dual-eq5}) provides the result.

(ii) As $\forall h \in H_c, s_{hc} \geq 0 $, using (\ref{dual-eq6}), it follows that $\forall c_a \in C_a, s_{c_a} - du^a_{c_a} \geq 0$. Then, we can pose $\tilde{du^a_{c_a}}:=0$ and make a change of variable $\tilde{s_{c_a}}:=s_{c_a} - du^a_{c_a}$ in (\ref{primal-eq1})-(\ref{cc-eq15}) (systems of conditions equivalent in the usual sense).
\end{proof}

\subsection{The new primal-dual framework}\label{subsec:newprimaldualframework}

This new 'primal-dual approach' makes use of an equality of objective functions (\ref{primaldual-eqobj}) to enforce all the economically meaningful complementarity conditions (\ref{cc-eq1})-(\ref{cc-eq15}). Leaving aside for the time being the question of MIC bid selections, which are dealt with in section \ref{section-mic}, the problem is that we don't know a priori what is the best block bid selection $J = J_r \cup J_a$. However, the feasible set UMFS described below allows to determine the optimal block bid selection, whatever the desired objective function is, and the pair of optimal points for the corresponding primal and dual programs stated above. This is formalised in Theorem \ref{maintheorem} and helps to consider many interesting issues (welfare or traded volume maximization, minimization of opportunity costs, etc), in a computationally efficient way. This is also a key step towards the main extension presented in next section, proposing an exact linearisation to deal with MIC bids using a MILP formulation.

In theory, admissible market clearing prices may lie outside the price range allowed for bids, see \cite{oneill2005}. 
For modelling purposes, we need to include the following technical constraint limiting the market price range
\begin{equation}\label{pricerangecondition}
\pi_{l,t} \in [-\bar{\pi}, \bar{\pi}] \qquad  \forall l\in L, t\in T.
\end{equation}
$\bar{\pi}$ can be choosen large enough to avoid excluding any relevant market clearing solution (see \cite{madani2014}). Note that in practice, power exchanges actually do impose that the computed prices $\pi_{l,t}$ stay within a given range in order to limit market power and price volatility.

\textbf{Uniform Market Clearing Feasible Set (UMFS):}

\begin{multline}\label{primaldual-eqobj}
\sum_{i} (\lambda^{i}P^{}_{i})x_i + \sum_{c, h\in H_c} (\lambda^{hc}P^{}_{hc})x_{hc}  +  \sum_{j,t} (\lambda^{j}P^{t}_{j})y_j  \\ \geq  \sum_{i} s_i + \sum_{j} s_j +\sum_c s_c  - \sum_{j \in J} d^a_{j} + \sum_m w_m v_m
\end{multline}

\begin{align}
&x_i \leq 1 & \forall i \in I  [s_i]\label{primaldual-eq1} \\
&y_j \leq 1 & \forall j \in J \ [s_j]\label{primaldual-eq2} \\
&x_{hc} \leq u_{c} & \forall h \in H_c, c\in C \ [s_{hc}] \label{primaldual-eq3} \\
&u_{c} \leq 1 &\forall c \in C [s_{c}] \label{primaldual-eq4} \\
&\sum_{i \in I_{lt}}P^{}_{i}x_i + \sum_{j \in J_l}P^{t}_{j}y_j  + \sum_{hc\in HC_{lt}} P^{}_{hc}x_{hc} \nonumber\\  & \hspace{ 4cm}= \sum_{k} e^k_{l,t} n_k, \ &  \forall (l,t) \ [\pi_{l,t}]\label{primaldual-eq5} \\
&\sum_{k} a_{m,k} n_k  \leq w_{m}\ & \forall m \in N \ \  [v_{m}]\label{primaldual-eq6}  \\                  
&x, y, u\geq 0, \label{primaldual-eq7} \\
&y,u \in \mathbb{Z} \label{primaldual-eq8} \\
&s_i + P^{}_i \pi_{l(i),t(i)} \geq P^{}_i \lambda^{i} , &  \forall i  \qquad [x_i] \label{primaldual-eq9} \\
&s_{hc} + P^{}_{hc}\pi_{l(hc),t(hc)} \geq P^{}_{hc}\lambda^{hc}, & \forall h \in H_c, c \ [x_{hc}] \label{primaldual-eq10} \\
&s_{j} + d^r_{j} - d^a_{j} + \sum_t P^{t}_{j}\pi_{l(j),t} \geq P^{}_{j}\lambda^{j}, & \forall j \in J \ [y_{j}] \label{primaldual-eq11}\\
 &(s_{c} + du^r_{c}) \geq \sum_{h \in H_{c}} s_{hc} & \forall c \in C [u_{c}] \label{primaldual-eq12} \\
&d^r_{j} \leq M_j (1-y_j) &\forall j \in J \label{primaldual-eq13} 
\end{align}
\begin{align}
&du^r_{c} \leq M_c (1-u_{c}) &\forall c\in C  \label{primaldual-eq14} \\
&d^a_{j} \leq M_j\ y_j &\forall j \in J \label{primaldual-eq15} \\
&du^a_{c} =0 &\forall c\in C \label{primaldual-eq16} \\
&\sum_m a_{m,k} v_m- \sum_{l,t}e^k_{l,t} \pi_{l,t} = 0 &\ \forall k \in K  [n_k] \label{primaldual-eq17}\\
&s_i, s_j,s_c,s_{hc}, d^a,d^r, du^a, du^r, v_m \geq 0 \label{primaldual-eq18}
\end{align}

Constants $M_j$ are choosen large enough in Constraints (\ref{primaldual-eq13}), (\ref{primaldual-eq15}) so that Constraints (\ref{primaldual-eq11}) are not restraining the range $[-\bar{\pi}, \bar{\pi}]$ of possible values for $\pi_{l,t}$ (or the possibility to paradoxically reject and accept block bids). They must indeed correspond to the maximum opportunity cost in conditions (\ref{primaldual-eq13}), or loss in conditions (\ref{primaldual-eq15}), that could be incurred to a block bid, for clearing prices in the allowed range. For this purpose, assuming that both the bid and market clearing prices satisfy (\ref{pricerangecondition}), it is sufficient to set $M_j:=K\sum_{t}|P^{t}_{j}|$ with $K=2\bar{\pi}$. This value of $K$ could be improved. For example, in constraints (\ref{primaldual-eq13}), one could set $K:=(\bar{\pi} - \lambda^j)$ for a sell block bid, and $K:=(\lambda^j - (-\bar{\pi}))$ for a buy block bid. Values of the $M_c$ are determined similarly with respect to (\ref{primaldual-eq14}) and condition (\ref{primaldual-eq12}): they must be such that the value of $s_c$ can be null whatever the values of the $s_{hc}$ are. Economically, the surplus $s_c$ of a rejected MIC bid will be null even if the \emph{potential} surpluses $s_{hc}$ of its suborders are not.

Also, we have made use of Lemma \ref{lemma-duadur} to set, without loss of generality, all the variables $du^a_c:=0$ in UMFS. This is clarified in the proof of Theorem \ref{maintheorem}.

\begin{theorem}\label{maintheorem}
(I) Let $(x,y,u,n,\pi,v,s,d^a,d^r, du^a, du^r)$ be any feasible point of UMFS satisfying the price range condition (\ref{pricerangecondition}), and let us define $J_r=\{j| y_j=0\},J_a=\{j | y_j=1\}, C_r=\{c| u_c=0\},C_a=\{c | u_c=1\}$. 

Then the projection $(x,y,u,n,\pi,v,s,d^a_{j_a \in J_a},d^r_{j_r \in J_r}, du^a_{c_a \in C_a}, du^r_{c_r \in C_r})$ satisfies all conditions in (\ref{primal-eq1})-(\ref{cc-eq15}). 

(II) Conversely, any point 

$MCS = (x,y,u,n,\pi,v,s,d^a_{j_a \in J_a},d^r_{j_r \in J_r}, du^a_{c_a \in C_a}, du^r_{c_r \in C_r})$ feasible for constraints (\ref{primal-eq1})-(\ref{cc-eq15}) related to a given arbitrary block order selection $J=J_r\cup J_a$ and MIC selection $C = C_r \cup C_a$ which respects the price range condition (\ref{pricerangecondition}) can be ‘lifted’ to obtain a feasible point $\tilde{MCS} = (x,y,u,n,\pi,v,\tilde{s},\tilde{d^a},\tilde{d^r}, \tilde{du^a}, \tilde{du^r})$ of UMFS.
\end{theorem}

\begin{proof}
See appendix.
\end{proof}

\section{Including MIC bids}\label{section-mic}

\subsection{Complex orders with a minimum income condition}\label{subsec:mic-description}

A MIC order is basically a set of hourly orders with the classical clearing rules but with the additional condition that a given 'minimum income condition' must be satisfied. Otherwise, all hourly bids associated to the given MIC bid are rejected, even if some of them are ITM. The minimum income condition of the MIC order $c$ ensures that some fixed cost $F_c$ together with a variable cost $V_c \times P_c$ are recovered, where $P_c$ is the total executed quantity related to the order $c$, and $V_c$ a given variable cost.

With the notation described in Section \ref{sec:notation}, the minimum income condition for a MIC bid indiced by $c$ has the form:

\begin{equation} \label{mic-eq0}
(u_c=1) \Longrightarrow \sum_{h \in H_c} (-P^{}_{hc}x_{hc}) \pi_{l(hc),t(hc)} \geq F_c + \sum_{h \in H_c} (-P^{}_{hc}x_{hc}) V_c,
\end{equation}

where $H_c$ denotes the set of hourly orders associated with the MIC order $c$. The left-hand side represents the total income related to order $c$, given the market prices $\pi_{l,t}$ and executed amount of power  $\sum_{h \in H_c} -P_{hc}x_{hc}$, while the right-hand side corresponds to the fixed and variable costs of production. At first sight, this condition is non-linear and non-convex, because of the terms $x_{hc} \pi_{l(hc),t(hc)}$ in the left-hand side.  

In previous works, MIC constraints are either approximated, see  \cite{GarciaBertrand,ruiz2011}, or the full model is decomposed and solved heuristically, as in the approach described in EUPHEMIA (\cite{euphemia}).
More specifically only the primal part \eqref{primal-obj}--\eqref{primal-eq8} is considered in a master problem, and prices are computed only when integer solutions are found. 
Let us however recall that \cite{GarciaBertrand,ruiz2011} on the one hand and \cite{euphemia} on the other hand are considering distinct market models. We show in the next section how MIC conditions can be linearized without approximation in the common European market model considered by \cite{euphemia}. To the best of our knowledge, this is the first exact linearisation proposed for this type of conditions.

\subsection{Exact linearization of the MIC conditions}

The following lemma is the key reason for which it is possible to express the a priori non-linear non-convex MIC condition (\ref{mic-eq0}) as a linear constraint. As we dispose in UMFS of both the surplus variables $s_c$ and the contributions to welfare $(P_{hc}x_{hc})\lambda^{hc}$, we can use them to express the income in a linear way:
\medskip

\begin{lemma}\label{lemma-mic-lin} Consider any feasible point of UMFS. Then, the following holds:
\begin{equation}\label{mic-eq1}
\forall c \in C, \sum_{h \in H_{c}} (P_{hc}x_{hc})\pi_{l(hc),t(hc)} = \sum_{h \in H_{c}} (P_{hc}\lambda^{hc}) x_{hc} - s_{c}
\end{equation} 
\end{lemma}

\begin{proof} We first define $C_r$ and $C_a$ as in Theorem \ref{maintheorem}.
For $c_r \in C_r$, the identity is trivially satisfied, because if a MIC bid is rejected, all related hourly bids are rejected: $\forall h \in H_{c_r}, x_{hc_r}=0$, and on the other side, $s_{c_r}=0$ because of complementarity constraints (\ref{cc-eq4}).

Let us now consider an accepted MIC bid $c_a \in C_a$. 
We first show that the following identity holds:

\begin{equation}
(P_{hc_a}x_{hc_a})\pi_{l(hc_a),t(hc_a)} = (P_{hc_a}\lambda^{hc_a}) x_{hc_a} - s_{hc_a} \label{linear}
\end{equation}

Consider for $x_{hc_a}$ the following two possibilities, noting that $u_{c_a}=1$:

(a) if $x_{hc_a} = 0$, the identity \eqref{linear} is trivially satisfied, as $s_{hc_a}=0$ according to complementarity constraints (\ref{cc-eq3}).

(b) if $0<x_{hc_a}$, (\ref{cc-eq11}) gives 
$s_{hc_a} = P_{hc_a}\lambda^{hc_a} - P_{hc_a}\pi_{l(hc_a),t(hc_a)}$, so multiplying the equation by $x_{hc_a}$ and using (\ref{cc-eq3}) guaranteeing $s_{hc_a} x_{hc_a} = s_{hc_a} u_{c_a} = s_{hc_a}$, we get identity \eqref{linear}.

Summing up \eqref{linear} over $h \in H_{c_a}$ yields: 
\begin{multline*}
\sum_{h \in H_{c_a}} (P_{hc_a}x_{hc_a})\pi_{l(hc_a),t(hc_a)} = \sum_{h \in H_{c_a}} (P_{hc_a}\lambda^{hc_a}) x_{hc_a} - \sum_{hc_a} s_{hc_a}
\end{multline*}

Finally, using complementarity constraints (\ref{cc-eq15}), we get:

\begin{multline*}
\sum_{h \in H_{c_a}} (P_{hc_a}x_{hc_a})\pi_{l(hc_a),t(hc_a)} = \sum_{h \in H_{c_a}} (P_{hc_a}\lambda^{hc_a}) x_{hc_a} - (s_{c_a} - du^a_{c_a}),
\end{multline*} 

where $du^a_{c_a}:=0$ by Lemma \ref{lemma-duadur} in the definition of UMFS, providing the required identity (\ref{mic-eq1}).\end{proof}

Using  Lemma \ref{lemma-mic-lin}, the MIC condition (\ref{mic-eq0}) can be stated in a linear way as follows:

\begin{equation}\label{mic-eq2}
s_c - \sum_{h \in H_{c}} (P_{hc}\lambda^{hc}) x_{hc} \geq F_c + \sum_{h \in H_{c}} (-P_{hc}x_{hc}) V_c -  \overline{M_c}(1-u_c)
\end{equation}

where $\overline{M_c}$ is a fixed number large enough to deactivate the constraint when $u_c=0$. As $u_c=0$ implies $s_c=0$ and $x_{hc}=0$, we set $\overline{M_c}:=F_c$.

\subsection{Welfare maximization with MIC bids, without any auxiliary variables}

We propose here a formulation of the welfare maximization problem including MIC bids, avoiding any auxiliary variables, by eliminating the variables $d^a, d^r, du^r$ from the formulation UMFS.

Let us consider UMFS with the additional MIC conditions (\ref{mic-eq2}) for $c \in C$. We can make a first simplification of the model by replacing both kinds of conditions (\ref{primaldual-eq12}), (\ref{primaldual-eq14}) by the conditions (\ref{clean-1}) below. Also, under PCR market rules, as no bid can be paradoxically executed, according to Lemma \ref{lemma-dadr}, we must set $ d^a_j = 0, \forall j\in J$. With this last conditions added, in the same way, we can clean up the mathematical formulation by replacing (\ref{primaldual-eq11}) and (\ref{primaldual-eq13}) by conditions (\ref{clean-2}) below, as well as removing constraints (\ref{primaldual-eq15})-(\ref{primaldual-eq16}). This yields an equivalent MILP formulation without any auxiliary variables, and in particular no more binary variables than the number of block and MIC bids:

\begin{align}   
&\textbf{PCR-FS} \nonumber\\
&\sum_{i} (\lambda^{i}P^{}_{i})x_i + \sum_{c, h\in H_c} (\lambda^{hc}P^{}_{hc})x_{hc}  +  \sum_{j,t} (\lambda^{j}P^{}_{j,t})y_j  \nonumber\\ & \geq  \sum_{i} s_i + \sum_{j} s_j +\sum_c s_c  + \sum_m w_m v_m \\
&x_i \leq 1 & \forall i \in I \  [s_i] \\
&y_j \leq 1 & \forall j \in J \ [s_j] \\
&x_{hc} \leq u_{c} & \forall h \in H_c, c\in C \ [s_{hc}]  \\
&u_{c} \leq 1 &\forall c \in C [s_{c}]  \\
&\sum_{i \in I_{lt}}P^{}_{i}x_i + \sum_{j \in J_l}P^{t}_{j}y_j  + \sum_{hc\in HC_{lt}} P^{}_{hc}x_{hc}  \nonumber\\ & \hspace{4cm} = \sum_{k} e^k_{l,t} n_k, \ &  \forall (l,t) \ \ [\pi_{l,t}] \\
&\sum_{k} a_{m,k} n_k  \leq w_{m}\ & \forall m \in N \ \  [v_{m}]  \\                  
&x, y, u\geq 0, \\
&y,u \in \mathbb{Z}  \\
&s_i + P^{}_i \pi_{l(i),t(i)} \geq P^{}_i \lambda^{i} , &  \forall i  \qquad [x_i]  \\
&s_{hc} + P^{}_{hc}\pi_{l(hc),t(hc)} \geq P^{}_{hc}\lambda^{hc}, & \forall h \in H_c, c \ [x_{hc}]  \\
&s_{j}  + M_j (1-y_j) + \sum_t P^{t}_{j}\pi_{l(j),t} \geq P^{}_{j}\lambda^{j}, & \forall j \in J \ [y_{j}] \label{clean-2}  \\
&s_{c} + M_c (1-u_{c}) \geq \sum_{h \in H_{c}} s_{hc} & \forall c \in C [u_{c}] \label{clean-1} \\
&\sum_m a_{m,k} v_m- \sum_{l,t}e^k_{l,t} \pi_{l,t} = 0 &\ \forall k \in K  [n_k] \\
&s_c - \sum_{h \in H_{c}} (P_{hc}\lambda^{hc}) x_{hc} \geq \nonumber\\ 
& \hspace{1cm} F_c + \sum_{h \in H_{c}} (-P_{hc}x_{hc}) V_c - \overline{M_c}(1-u_c) &\ \ \forall c \in C \\
&s_i, s_j,s_c,s_{hc}, d^r, du^r, v_m \geq 0 
\end{align}

The welfare optimization problem is then stated as follows:

\begin{equation}
\max_{PCR-FS} \ \  \sum_{i} (\lambda^{i}P^{}_{i})x_i + \sum_{c, h\in H_c} (\lambda^{hc}P^{}_{hc})x_{hc}  +  \sum_{j,t} (\lambda^{j}P^{}_{j,t})y_j
\end{equation}

\section{Using the framework for economic analysis purposes}\label{section-usingframework}

\subsection{Maximizing the traded volume} The following program aims at maximizing the traded volume under the same market rules:

\begin{center}
\begin{multline}
\max_{PCR-FS} \ \  \sum_{i | P_i >0} P^{}_{i}x_i + \sum_{c, h\in H_c | P_{hc} > 0} P^{}_{hc}x_{hc}  +  \sum_{(j,t) | P^{t}_{j} > 0} P^{t}_{j}y_j 
\end{multline}
\end{center}

An alternative formulation of the objective function is the following:

\begin{center}
\begin{multline}
\max_{PCR-FS} \ \ \frac{1}{2}(  \sum_{i} |P^{}_{i}|x_i + \sum_{c, h\in H_c} |P^{}_{hc}|x_{hc}  +  \sum_{(j,t) } |P^{t}_{j}|y_j) 
\end{multline}
\end{center}

\subsection{Minimizing opportunity costs of PRB}

It suffices to consider the following objective function over UMFS and the additional constraints $d^a_j = 0, \forall j \ J$:

\begin{center}
\begin{equation}
\min \sum_j d^r_j
\end{equation}
\end{center}

Let us also note that constraints like (\ref{primaldual-eq13}) also allow to control which block bids could be paradoxically rejected on an individual basis, or could also be used to forbid the paradoxical rejection of bids which are too deeply in-the-money, by specifying a threshold via the values of $M_j$.

\subsection{A short proof for a property of an alternative market model}

The result below, first proposed in \cite{oneill2007}, basically states that (a) there is always enough welfare to allow and compensate paradoxically accepted block bids (PAB), (b) allowing PAB generates globally more welfare. We show how this can be derived almost directly from the proposed framework, and we specify the decomposition of the welfare into the sum of (non-negative) individual bid surpluses minus paid compensations to PAB, cf. the interpretation of the $d^a_j$ given in Lemma \ref{lemma-dadr}. Therefore, there is an incentive for market participants, globally, to allow and compensate PAB:

\begin{theorem}
Let us denote UMFS2 the feasible set described by the constraints of UMFS (\ref{primaldual-eqobj})-(\ref{primaldual-eq18}), together with the MIC constraints (\ref{mic-eq2}), and PCR-FS2 the same feasible set UMFS2 with the additional constraints $\forall j \in J, d^a_j=0$ to forbid PAB. Then, trivially:
\begin{enumerate}
\item For any feasible point of UMFS2 (resp. PCR-FS2):

\begin{multline}
\sum_{i} (\lambda^{i}P^{}_{i})x_i + \sum_{c, h\in H_c} (\lambda^{hc}P^{}_{hc})x_{hc}  +  \sum_{j,t} (\lambda^{j}P^{}_{j,t})y_j  \\ =  \sum_{i} s_i + \sum_{j} s_j +\sum_c s_c + \sum_m w_m v_m  - \sum_{j \in J} d^a_{j} 
\end{multline}

\item 
\begin{multline} 0\leq \max_{PCR-FS2} \sum_{i} s_i + \sum_{j} s_j +\sum_c s_c + \sum_m w_m v_m
 \\ \leq \max_{UMFS2} \sum_{i} s_i + \sum_{j} s_j +\sum_c s_c + \sum_m w_m v_m  - \sum_{j \in J} d^a_{j} 
\end{multline}
\end{enumerate}

\end{theorem}

\begin{proof}
\begin{enumerate}
\item The first equality is a consequence of duality for linear programs.

\item As $PCR-FS2 \subseteq UMFS2$, and in $PCR-FS2$, $\forall j \in J, d^a_j=0$, the inequality follows.
\end{enumerate}
\end{proof}

\section{Numerical Experiments}\label{sec-numerical}

We provide here a proof-of-concept of the approach, presenting numerical experiments, using realistic \emph{large-scale instances}, on: (a) welfare maximization with MIC bids, (b) traded volume maximisation under CWE rules (i.e. without MIC bids), and (c) opportunity costs minimization also under CWE market rules. The models have been implemented in C++ using IBM ILOG Concert Technology interfaced with R for input-output management as well as post-processing analysis, and solved using CPLEX 12.5.1 using 4 threads on a platform with 2x Xeon X5650 (6 cores @ 2.66 GHz), 16 GB of RAM, running Fedora Linux 20. One potential advantage of the new primal-dual approach is the possibility to benefit from parallel computing routines of state-of-the-art solvers like CPLEX, without implementing further algorithmic work.

\subsection{Welfare maximization with MIC bids}

We first consider solving market instances with MIC bids only. The instances involve hourly bids from four areas, one of these areas also containing MIC bids. Solving this kind of instances up to optimality is tractable, see Table \ref{table-welfaremic}, though some instances are challenging from a numerical stability point of view, due to the introduction of BigM's. A particular care should be given to solver's tolerance parameters, e.g. integer feasibility tolerance. Heuristics and cuts have been here deactivated, and the branching direction has been set to -1 (priority to the 'down branch'). The idea is that priority must be given to eliminating as many MIC bids as needed to increase the prices sufficiently to ensure enough income to the accepted ones. It should be noted that when discarding the MIC conditions, no MIC bid is getting the minimum income required, certainly because in that (non-feasible) case, the overall offer formed by the hourly suborders is too important at low marginal prices.

Let us emphasize that this new formulation provides dual LP bounds, and can be solved \emph{exactly} by state-of-the-art MIP solvers. To our knowledge, this is the first tractable exact approach proposed. (Another, apparently non-tractable approach would be to proceed by decomposing the problem and adding e.g. no-good cuts rejecting the current MIC bids selection when no admissible prices exist, with a very slow convergence rate.)

\begin{table}[h]
\begin{center}
\begin{tabular}{c|cccccc}
\#inst	&Run. Time (s.)	&Nodes	&Abs. Gap	&Rel. Gap	&\#Houly bids	&\# MIC bids\\
\hline
1	&70	&58	&	&	&47107	&70\\
2	&274	&896	&	&	&49299	&74\\
3	&432	&1111	&	&	&48119	&71\\
4	&144	&104	&	&	&52434	&72\\
5	&37	&19	&	&	&41623	&74\\
6	&901	&3238	&1299754.97	&0.04\%	&45371	&69\\
7	&22	&23	&	&	&36819	&73\\
8	&624	&1055	&	&	&53516	&69\\
9	&255	&504	&	&	&62770	&76\\
10	&216	&418	&	&	&45731	&74\\
\end{tabular}
\end{center}
\caption{Welfare optimization with MIC bids}
\label{table-welfaremic}
\end{table}

Adding block bids render the problem much more difficult to solve. Solver's parameters used are the same as above. It turns out that many block bids are fractionally accepted in continuous relaxations of the branch-and-cut tree. As a consequence, the first feasible solutions found are of poor quality, with a few block and MIC bids accepted. Instead, as an easy-to-implement heuristic approach, we first solve an instance with all block bids fixed to zero, and determine an admissible MIC bids selection (ideally optimal for this subproblem). Second, we fix this MIC bids selection, and introduce block bids, to determine a potentially very good solution to the initial problem. It should be noted that a solution for this second stage always exists, as in the worst case, all block bids could be rejected. Third, the obtained solution is used as a MIP start for the initial model with both block and MIC bids. With this approach, the number of block bids accepted is of the same order as when no MIC bids are considered besides, and the relative gap is improved, compared to the basic approach of providing the solver with the formulation 'as is'. However the absolute MIP gap remains substantial. Results are presented in Table \ref{table-welfaremic_wb}, with a running time limit of 900 seconds for each of the two first  stages, and 1200 seconds for the last stage.
\begin{table}[h]
\begin{center}
\begin{tabular}{c|cccccc}
\#inst	&Nodes	&Abs. Gap	&Rel. Gap	&\#Houly bids	&\# MIC bids	&\# Blocks bids\\
\hline
1	&13214	&1015887.92	&0.04\%	&47107	&70	&502\\
2	&7913	&4129620.69	&0.14\%	&49299	&74	&589\\
3	&11375	&2748987.16	&0.12\%	&48119	&71	&516\\
4	&2873	&3009748.14	&0.10\%	&52434	&72	&591\\
5	&12213	&1425671.83	&0.05\%	&41623	&74	&588\\
6	&6443	&5999741.05	&0.19\%	&45371	&69	&567\\
7	&22250	&337651.70	&0.01\%	&36819	&73	&550\\
8	&6925	&4747440.57	&0.19\%	&53516	&69	&691\\
9	&3658	&2937928.67	&0.08\%	&62770	&76	&604\\
10	&3194	&3100317.15	&0.12\%	&45731	&74	&537\\\end{tabular}
\end{center}
\caption{Welfare optimization with MIC and block bids}
\label{table-welfaremic_wb}
\end{table}

\subsection{Traded volume maximization}
To optimize the traded volume, welfare maximization itself turns out to be a useful \emph{heuristic}. We first solve this welfare maximization problem, and for the given optimal block bid selection, maximize the traded volume (dealing with a possible indeterminacy of the traded volume for that solution). 
We then use this solution as a MIP start for the initial problem. This helps in practice, at least to provide a useful upper bound, even in some cases proving optimality of the welfare maximizing solution for the traded volume maximization problem. We also observed that  provided the reasonably good solution obtained from maximizing welfare, well-known heuristics such as 'solution polishing' in CPLEX could provide quickly better solutions. Therefore, this heuristic is first applied when considering the second stage solving the initial traded volume maximization problem itself, provided the solution obtained at the first stage. Table \ref{table-tradedvolume} summarizes the trade-off  between both kinds of objectives for ten instances corresponding to the belgian market.
\begin{table}[h]
\begin{center}
\begin{tabular}{c|ccccccc}
	&Welf. max sol.	&\multicolumn{2}{c}{Maximizing traded volume}	&$\Delta$ Vol.	& $\Delta$ Welf. &\# Hourly & \# Blocks \\
\#	&Max Trad. Vol.	&Max Trad. Vol.	&Best bound	& & &bids &bids \\
\hline
1	&24589.84	&24589.84	&24589.84	&0.00	&0.00	&1939	&54\\
2	&25794.53	&25928.19	&26654.32	&133.66	&5672.35	&1711	&67\\
3	&23633.48	&23696.99	&23696.99	&63.51	&171.40	&1706	&54\\
4	&35137.32	&35285.32	&35285.32	&148.00	&4292.13	&1893	&56\\
5	&21296.94	&21433.08	&21433.08	&136.15	&2460.55	&1713	&39\\
6	&23361.72	&23871.27	&23871.27	&509.55	&56518.94	&1700	&46\\
7	&23542.38	&23679.64	&23679.64	&137.26	&1877.94	&1749	&35\\
8	&35974.15	&36270.11	&36270.11	&295.96	&21403.03	&1533	&58\\
9	&24988.63	&24988.63	&24988.63	&0.00	&0.00	&1787	&33\\
10	&35307.91	&35507.32	&37434.39	&199.41	&16082.15	&1418	&62\\
\end{tabular}
\end{center}
\caption{Comparison of the maximum traded volume in both cases}
\label{table-tradedvolume}
\end{table}

For example, instances \# 2 or \# 3 show concrete examples where it is possible to obtain more traded volume than when just optimizing welfare (as in the toy example presented above), the better solution for \# 3  even being proven optimal. Instance \# 1 shows an example where the welfare maximizing solution is proven optimal for the traded volume maximization problem. Sometimes the traded volume can be significantly larger (2\% or more), as in instance \#6.

\subsection{Minimizing opportunity costs}

We proceed as above, (a) first solving the welfare maximizing solution, (b) looking for the minimum opportunity costs possible for this solution, and (c) use this solution as a start solution for the proper opportunity costs minimization problem. Let us note that prices and opportunity costs obtained from stage (b) can substantially differ from the prices computed in practice, as these prices are determined in a different way from what is specified by tie breaking rules in case of price/volume indeterminacy. Let us recall that welfare is uniquely determined by the block bid selections and hence not affected by stage (b), see \cite{madani2014}. We also refer to \cite{madani2} for a table showing results for a few real CWE instances from 2011.

Results are given in Table \ref{table-opportunitycost}. Optimal solutions are found in the majority of the instances (9 out of 10).
Again, for example, instances \# 1 or \# 2 show that solutions to both problems do not coincide in general (as in the toy example of section \ref{subsec:toyexample}). Opportunity cost can sometimes be reduced by 75\% or more, for example for instance \#3. 
In the case of instances \# 5 and \# 9, the welfare maximizing solution is proven optimal for the opportunity costs minimization problem.

\begin{table}[h]
\begin{center}
\begin{tabular}{c|ccccccc}
	&W-MAXSOL	&\multicolumn{2}{c}{Minimizing Opp. Costs}	&$\Delta$ OC.	& $\Delta$ Welf. &\# Hourly & \# Blocks \\
\#	&Min OC	& Min OC	&Best bound	& & & & \\
\hline
1	&13096.96	&5624.37	&5624.37	&7472.58	&2501.76	&1939	&54\\
2	&6559.97	&2124.96	&2124.96	&4435.01	&963.19	&1711	&67\\
3	&3913.16	&978.61	&978.61	&2934.55	&171.40	&1706	&54\\
4	&483.71	&348.00	&348.00	&135.71	&138.65	&1893	&56\\
5	&1715.30	&1715.30	&1715.30	&0.00	&0.00	&1713	&39\\
6	&49366.33	&46405.44	&46405.44	&2960.90	&1577.61	&1700	&46\\
7	&8771.51	&8771.51	&8771.51	&0.00	&0.00	&1749	&35\\
8	&17249.96	&7399.43	&7399.43	&9850.53	&236.38	&1533	&58\\
9	&256.81	&256.81	&256.81	&0.00	&0.00	&1787	&33\\
10	&64777.46	&61579.08	&3198.25	&3198.37	&1591.57	&1418	&62\\
\end{tabular}
\end{center}
\caption{Comparison of opportunity costs in both cases}
\label{table-opportunitycost}
\end{table}

\section{Conclusions}

The new primal-dual approach proposed here allows to derive powerful algorithmic tools as well as enables to deal with economic issues of interest for day-ahead auctions stakeholders. We have been able to give a MILP formulation of the market clearing problem in the presence of MIC bids, avoiding the introduction of any auxiliary variables, relying on an exact linearisation of the minimum income condition. To the best of our knowledge, it is the first tractable exact approach proposed to deal with such kind of bids, and numerical experiments show good results, though the approach is still challenging when both block and MIC bids are considered together. From the economic analysis point of view, the approach allowed us to examine the trade-off occurring in practice between different objectives such as welfare maximization, traded volume maximization, and minimization of opportunity costs of paradoxically rejected block bids. It also seems  these are the  first tractable formulations proposed to examine these economic issues. 
The trade-offs for the examined instances were rather small, though they could be more important in absolute terms if the number and size of non-convex bids are allowed to increase. 
We also plan to release soon a Julia package implementing the models and algorithms, to foster exchanges and provide adaptable tools to the academic community working on related research topics.

\textbf{Acknowledgements:}
We greatly thank APX, BELPEX, EPEX SPOT, OMIE and N-Side for providing us with data used to generate realistic instances. We also thank organizers of the 11th International (IEEE) Conference on European Energy Market (Krakow, May 2014), as well as organizers of the COST Workshop on Mathematical Models and Methods for Energy Optimization (Budapest, Sept. 2014), for allowing us to present partial results developed here. This text presents research results of the P7/36 PAI project COMEX, part of the IPA Belgian Program. The work was also supported by EC-FP7 COST Action TD1207. The scientific responsibility is assumed by the authors.


%
%

\section{Omitted proofs in main text}

Let us first consider the equality of primal an dual objective functions of section \ref{subsec:unrestricted-welfare}:

\begin{observation}\label{observation-1}
By strong duality for linear programs, for a pair of primal and dual feasible points corresponding to a block bid selection and a MIC bid selection, i.e. satisfying respectively (\ref{primal-eq1})-(\ref{fix-eq4}) and (\ref{dual-eq1})-(\ref{dual-eq8}), the complementarity constraints (\ref{cc-eq1})-(\ref{cc-eq15}) hold if and only if we have the equality $(\ref{primal-obj}) = (\ref{dual-obj})$.
\end{observation}

\subsection{Proof of Theorem \ref{maintheorem}}

\begin{proof}
We emphasize again, and use below, the fact that according to Lemma \ref{lemma-duadur}, we can assume without loss of generality $du^a_{c_a} = 0, \forall
c_a \in C_a$ in (\ref{dual-obj})-(\ref{dual-eq8}).

(I) Let $MCS=(x,y,n,u,\pi,v, s,d^a,d^r, du^a, du^r)$ be a feasible point of UMFS and let us define $J_r := \{j \in J | y_j = 0\}, J_a := \{j \in J | y_j = 1\}$ and likewise for $C_r, C_a$ with respect to the values of the variables $u_c$. Consider the projection $\tilde{MCS}=(x,y,n,u, \pi,v, s,d^a_{j_a \in J_a},d^r_{j_r \in J_r}, du^a_{c_a \in C_a}, du^r_{c_r \in C_r})$. Constraints (\ref{primaldual-eq11})-(\ref{primaldual-eq16}) ensure that $\tilde{MCS}$ satisfies constraints (\ref{dual-eq3})-(\ref{dual-eq6}): constraints (\ref{primaldual-eq13})-(\ref{primaldual-eq16}) are 'dispatching' constraints (\ref{primaldual-eq11})-(\ref{primaldual-eq12}) to constraints (\ref{dual-eq3})-(\ref{dual-eq6}). Therefore $\tilde{MCS}$ satisfies primal conditions (\ref{primal-eq1})-(\ref{fix-eq4}) and dual conditions (\ref{dual-eq1})-(\ref{dual-eq8}). Condition (\ref{primaldual-eq15}) ensures that $d^a_j=0$ for $j \in J_r$, and with (\ref{primaldual-eq16}), it shows that condition (\ref{primaldual-eqobj}) implies the equality $(\ref{primal-obj}) = (\ref{dual-obj})$. By Observation 1, we can then replace this equality by the needed complementarity conditions (\ref{cc-eq1})-(\ref{cc-eq15}).

(II) Conversely let $\tilde{MCS}=(x,y,n,u, \pi,v, s,d^a_{j_a \in J_a},d^r_{j_r \in J_r}, du^a_{c_a \in C_a}, du^r_{c_r \in C_r})$ be a point satisfying constraints primal conditions (\ref{primal-eq1})-(\ref{fix-eq4}), dual conditions (\ref{dual-eq1})-(\ref{dual-eq8}), and complementarity conditions (\ref{cc-eq1})-(\ref{cc-eq15}),  associated to a given block and MIC bid selection $J=J_a \cup J_r, C = C_a \cup C_r$. Observation 1 ensures that this point also satisfies the equality $(\ref{primal-obj}) = (\ref{dual-obj})$. Let us set additional values $d^r_j=0$, for $ j \in J_a$, also $d^a_j=0$ for $j \in J_r$, and similarly $du^a_c = 0$ for $c \in C_r$, $du^r_c = 0$ for $c \in C_a$, giving a point $MCS=(x,y,n, u, \pi,v, s,d^a,d^r, du^a, du^r)$.  The new point satisfies condition (\ref{primaldual-eqobj}), since only terms $d^a_j=0, j \in J_r, du^a_c=0$ are added to the equality $(\ref{primal-obj}) = (\ref{dual-obj})$. It remains to verify that all the remaining constraints defining UMFS are satisfied as well. All these additional values trivially satisfy constraints (\ref{primaldual-eq13})-(\ref{primaldual-eq16}).  Therefore, it is needed to show that conditions (\ref{primaldual-eq11})-(\ref{primaldual-eq16}) are also satisfied \emph{for all} $j \in J, c \in C$. Due to (\ref{cc-eq2}), in condition (\ref{dual-eq3}), $s_{j_r}=0$ and we can set $d^r_{j_r }:=P_{j_r} \lambda^{j_r}-P_{j_r} \pi$  without altering the satisfaction of any condition. Due to the price range condition and the choice of the parameters $M_j$, these $d^r_j,j \in J_r$ satisfy conditions (\ref{primaldual-eq13}), therefore satisfied for all $j \in J$. In condition (\ref{dual-eq4}), $s_{j_a},d^a_{j_a}$ can be redefined without modifying the values of $(s_{j_a}-d^a_{j_a})$ and hence without altering satisfaction of any other constraint. Due to the large values of the parameters $M_j$, this again can be done so as to satisfy conditions (\ref{primaldual-eq15}) for $j \in J_a$, hence for all $j \in J$. Then, (\ref{dual-eq3})-(\ref{dual-eq4}), and the 'dispatcher conditions' (\ref{primaldual-eq13})-(\ref{primaldual-eq15}) imply (\ref{primaldual-eq11}). Finally, concerning the analogue constraints related to the MIC bids, and first using Lemma \ref{lemma-dadr} to set $du^a_{c_a} = 0$ for all $c_a \in Ca$,  it is straightforward to show in a similar way that (\ref{dual-eq5})-(\ref{dual-eq6}) together with the $M_c$ and the additional null values for part of the $du^r$ (resp. $du^a$) given above allow to satisfy (\ref{primaldual-eq12}), (\ref{primaldual-eq14}).
\end{proof}

\bibliography{template}

\bibliographystyle{plainnat}


\end{document}